\documentclass[oneside]{article}
\usepackage{amsmath,amssymb,amscd,amsthm,verbatim,alltt,amsfonts,array}
\usepackage{mathrsfs}
\usepackage[english]{babel}
\usepackage{latexsym}
\usepackage{amssymb}
\usepackage{euscript}
\usepackage{graphicx}
\usepackage{arydshln}
\usepackage[dvipsnames]{xcolor}

\newtheorem{theorem}{Theorem}
\theoremstyle{plain}

\newtheorem{conjecture}{Conjecture}
\newtheorem{corollary}{Corollary}

\newtheorem{remark}{Remark}

\numberwithin{equation}{section}

\begin{document}

{\footnotesize%
\hfill
}

  \vskip 1.2 true cm

\begin{center} {\bf Arithmeticity for integral cohomological dimension of configuration spaces of manifolds} \\
          {by}\\
{\sc Muhammad Yameen}
\end{center}

\pagestyle{myheadings}
\markboth{Cohomological dimension of configuration spaces of manifolds}{Muhammad Yameen}

\begin{abstract}
Consider the configuration spaces of manifolds. We give a precise formula for the integral cohomological dimension (the degree of top non-trivial integral cohomology group) of unordered configuration spaces of manifolds with non-trivial co-dimension one cohomology group, and show that the the sequence of cohomological dimensions is arithmetic. This arithmeticity is not present in the classical example of Arnold. Moreover, we show that the top integral cohomology group is infinite. Furthermore, We give a lower bound for the rank of top integral cohomology group. We also predict that the top integral cohomology group of configuration spaces of manifolds with non-trivial co-dimension one cohomology group is eventually finite. To the best of our knowledge, there is no rigorous bound for the cohomological dimension of ordered configuration spaces. As an application of main results, we give a sharp lower bound for the cohomological dimension of ordered configuration spaces of manifolds. The first step of the proof of main results is to define a reduced Chevalley–Eilenberg complex. 
\end{abstract}

\begin{quotation}

\end{quotation}

\thispagestyle{empty}

\section{Introduction}

\label{sec:intro}

%%%%\begin{equation}   \label{eq:defconf}    \end{equation}  \begin{theorem}   \label{thm:main1}   \cite{S}

For any manifold $M$, let
$$F_{k}(M):=\{(x_{1},\ldots,x_{k})\in M^{k}| x_{i}\neq x_{j}\,for\,i\neq j\}$$
be the configuration space of $k$ distinct ordered points in $M$ with induced topology. The symmetric group $S_{k}$ acts on $F_{k}(M)$ by permuting the coordinates. The quotient $$C_{k}(M):=F_{k}(M)/S_{k} $$
is the unordered configuration space with quotient topology. The geometry and topology of configuration spaces have attracted a important attention, over the years. Configuration spaces is a vibrant area of research of geometry and topology with important applications in several areas including algebraic topology, geometric topology, algebraic geometry, knot theory, geometric group theory, differential topology or homotopy theory. These spaces are cornerstones objects in topology, source of several crucial topological data. For example, when the base manifold if the surface, the fundamental groups are the so-called braid groups which are fundamental in geometric group theory. If $M$ is a manifold of dimension $d$, then the spaces $F_{k}(M)$ and $C_{k}(M)$ are open manifolds of dimension $kd.$ When $M$ is an open surface, then $C_{k}(M)$ is a stein manifold. The manifold $C_{k}(M)$ is orientable if and only if $M$ is orientable and even dimensional.

The central theme in the theory of configuration spaces is the (co)homology of these spaces. The cohomology of these spaces was first studied by Arnold (see \cite{A1} and \cite{A2}) in his investigation of spaces $F_{k}(\mathbb{R}^{2})$ and $C_{k}(\mathbb{R}^{2}).$ The subject was revival and systematically investigated by Cohen and his collaborators. For an orientable manifold $M$, spectral sequence of the inclusion $F_{k}(M)\hookrightarrow  M^{k}$ has been described by Cohen-Taylor and further studied by Totaro \cite{T}. For a smooth complex projective variety $X,$ Fulton and MacPherson \cite{F-M} obtained a compactification of $F_{k}(X).$ Using Morgan's theory they describe a differential graded algebra model for the rational homotopy of $F_{k}(X),$ depending on $H^{*}(X;\mathbb{Q}),$ the number of points $k,$ and the Chern classes of $X.$ Furthermore, Fulton-MacPherson's model was simplified by Kriz \cite{K} (see also \cite{BMP}). For the integral (co)homology, very little is known, except a few cases, which we would like to mention here. In the paper \cite{FZ}, Fiechtner and Zeigler’s describe the integral cohomology algebra of configuration spaces of usual spheres. The integral cohomology of configuration spaces of surfaces was studied by Napolitano in \cite{N2} (see also \cite{N1}). However, there is no explicit method to compute the integral (co)homology of configuration spaces of manifolds. An important question about the configuration space of manifolds is cohomological dimension. The aim of this work is to start the rigorous investigation of top non-trivial integral cohomology of these spaces. We describe how the cohomological dimension of these spaces is affected by increasing the number of points. All manifolds are assumed to be finite type (Betti numbers are finite numbers). The boundary of manifold $M$ is written $\partial M.$ Throughout the paper, $U$ is a closed subset of $M$ such that $\partial M\cap U=\phi,$ and $M-U$ is connected. We denote by $Chdim(M)$ (cohomological dimension of orientable manifold $M$) the smallest integer with the property that 
$$H^{i}(M;\mathbb{Z})=0, \quad \forall \,\,i>Chdim(M).$$ 

\begin{theorem}\label{theoremmain1}
Let $M$ be a compact connected orientable even dimensional manifold . Assume $\partial M\cup U\neq\phi$ and $rank(H^{d-1}(M-U;\mathbb{Z}))\neq0.$ Then, for $k\geq1,$ we have $$ Chdim(C_{k}(M-U)) =(d-1)k.$$
Moreover, $H^{(d-1)k}(C_{k}(M-U);\mathbb{Z})$ is infinite.
\end{theorem}
\begin{theorem}\label{theoremmain2}
Let $M$ be a closed connected orientable even dimensional manifold. Assume $rank(H^{d-1}(M;\mathbb{Z}))\neq0,1.$ Then, for $k\geq3,$ we have $$ Chdim(C_{k}(M)) =(d-1)k+1.$$
Moreover, $H^{(d-1)k+1}(C_{k}(M);\mathbb{Z})$ is infinite.
\end{theorem}

\begin{remark} 
Arnold proved the repetition property of $H^{*}(C_{k}(\mathbb{R}^{2});\mathbb{Z}):$
$$H^{i}(C_{2k+1}(\mathbb{R}^{2});\mathbb{Z})=H^{i}(C_{2k}(\mathbb{R}^{2});\mathbb{Z}).$$ For small values of $k,$ Napolitano computed the cohomological dimensions of configuration spaces surfaces with low genus (see Tables 1, 2, 3 and 4 of \cite{N2}). It seems that the behavior of cohomological dimensions of configuration spaces of manifold $M$ with condition $rank(H^{d-1}(M;\mathbb{Z}))= 0$ is not arithmetic.
\end{remark}

 The behavior of the sequence of cohomological dimensions of configuration spaces of manifold $M$ with condition $rank(H^{d-1}(M;\mathbb{Z}))= 0$ is unclear, and we ask the following.\\\\
\textbf{Question.} If $M$ is an orientable even dimensional manifold with condition $rank(H^{d-1}(M;\mathbb{Z}))= 0$, then what is the pattern in the sequence $\{Chdim(C_{k}(M))\}_{k\geq1}$?
\begin{theorem}\label{theoremmain3}
Let $M$ be a compact orientable even dimensional manifold. Assume $k\geq 2,$ $\partial M \cup U\neq\phi$ and $rank(H^{d-1}(M-U;\mathbb{Z}))=n>0.$ Then the group $H^{(d-1)k}(C_{k}(M-U);\mathbb{Z})$ is infinite. More precisely, the rank of $H^{(d-1)k}(C_{k}(M-U);\mathbb{Z})$ is at least
$$\qquad\binom{\frac{k}{2}+n-1}{n-1}\qquad\text{if $k$ is even }, $$
$$n\binom{\frac{k-1}{2}+n-1}{n-1}\qquad \text{if $k$ is odd}.$$
\end{theorem}
\begin{theorem}\label{theoremmain4}
Let $M$ be a closed orientable even dimensional manifold. Assume $k\geq 3,$ and $rank(H^{d-1}(M;\mathbb{Z}))=n>1.$ Then the group $H^{(d-1)k+1}(C_{k}(M);\mathbb{Z})$ is infinite. More precisely, the rank of $H^{(d-1)k+1}(C_{k}(M);\mathbb{Z})$ is at least
$$n\binom{\frac{k}{2}+n-2}{n-1}-\binom{\frac{k}{2}+n-1}{n-1}\qquad \text{if $k$ is even }, $$
$$\binom{\frac{k-1}{2}+n-1}{n-1}\qquad\qquad\qquad\qquad \text{if $k$ is odd}.$$
\end{theorem}

\begin{remark}
From the classical example of Arnold, we see that the integral cohomology group $H^{ i}(C_{k}(\mathbb{R}^{2}),\mathbb{Z})$ is finite for $i>1.$ Also, D. B. Fuchs \cite{F} proved that for an arbitrary degree $i$, one can find a large $k$ such that $H^{\geq i}(C_{k}(\mathbb{R}^{2});\mathbb{Z}_{2})$ is nonzero. So, the top non-zero integral cohomology groups of $C_{k}(\mathbb{R}^{2})$ is finite for sufficiently large $k.$ 
\end{remark}
We formulate a finiteness conjecture for the top integral cohomology of configuration spaces of manifolds with condition $rank(H^{d-1}(M;\mathbb{Z}))= 0$:
\begin{conjecture}(Finiteness)
Let $M$ be an orientable even dimensional manifold. Assume $rank(H^{d-1}(M;\mathbb{Z}))$ is zero. Then there exist $k_{0}\in\mathbb{N}$ such that the cohomology group $H^{top}(C_{k}(M);\mathbb{Z})$ is finite for $k\geq k_{0}.$
\end{conjecture}

\subsection*{Outline of the paper}
The paper is split into five sections. In section 2, we give a quick tour of Chevalley–Eilenberg complex. In section 3, we define a reduced Chevalley–Eilenberg complex. Main results are proved in sections 4. In section 5, we give a sharp lower bound for the cohomological dimension of ordered configuration spaces of manifolds. 
\subsection*{General conventions}
$\bullet$ We work throughout with finite dimensional graded vector spaces. The degree of an element $v$ is written $deg(v)$.\\\\
$\bullet$ The symmetric algebra $Sym(V^{*})$ is the tensor product of a polynomial algebra and an exterior algebra:
$$ Sym(V^{*})=\bigoplus_{k\geq0}Sym^{k}(V^{*})=Poly(V^{even})\bigotimes Ext(V^{odd}), $$
where $Sym^{k}$ is generated by the monomials of length $k.$\\\\
$\bullet$ The $n$-th suspension of the graded vector space $V$ is the graded vector space $V[n]$ with
$V[n]_{i} = V_{i-n},$ and the element of $V[n]$ corresponding to $a\in V$ is denoted $s^{n}a;$ for example
$$ H_{*}(S^{2};\mathbb{Q})[n] =\begin{cases}
      \mathbb{Q}, & \text{if $*\in\{n,n+2 \}$} \\
      0, & \mbox{otherwise}.\\
   \end{cases} $$ \\\\
$\bullet$ We write $H_{-*}(M;\mathbb{Q})$ for the graded vector space whose degree $-i$ part is
the $i$-th homology group of $M;$ for example
$$ H_{-*}(\text{CP}^m;\mathbb{Q}) =\begin{cases}
      \mathbb{Q}, & \text{if $*\in\{-2m,-2m+2,\ldots,0. \}$} \\
      0, & \mbox{otherwise}.\\
   \end{cases} $$

%%%%%%%%%%%%%%%%%%%%%%%%%%%%%%%%%%%%%%%%%%%%%%%%%%%%%%%%%%%%%%%%%%%%%%%%%%%%%%%%%%%%%%%%%%%%%%%%%%%%%%
%%%%%%%%%%%%%%%%%%%%%%%%%%%%%%%%%%%%%%%%%%%%%%%%%%%%%%%%%%%%%%%%%%%%%%%%%%%%%%%%%%%%%%%%%%%%%%%%%%%%%%%%

\section{Chevalley–Eilenberg complex}
The cohomology and homology of configuration spaces have received a lot of attention, since the work by Arnold. Fulton--Macpherson \cite{F-M} described a model $F(k)$ for the cohomology of $F_{k}(X)$ of a smooth projective variety $X$, where $F(k)$ depends on the cohomology ring of $X$, the canonical orientation class and the Chern classes of $X$. A simplified version of the Fulton--MacPherson model obtained by Kriz \cite{K}. The kriz's model does not depend on Chern classes. Independently the work of Kriz, Totaro proof the same result. Totaro's idea was to study the Leray spectral sequence of the inclusion $F_{k}(X)\subset X^{k}.$ The natural action of the symmetric group on the configuration spaces $F_{k}(X)$
induces an action on the Kriz model. The cohomology of $C_{k}(X)$ is obtained by the $S_{k}$--invariant part of Fulton--MacPherson and Kri\v{z} models (see corollary 8c of \cite{F-M} and remark 1.3 of \cite{K}):$$H^{i}(C_{k}(X);\mathbb{Q})\approx H^{i}(F_{k}(X);\mathbb{Q})^{S_{k}}.$$  B\"{o}digheimer--Cohen--Taylor \cite{B-C-T} studied the homology of $C_{k}(M)$ for the odd dimensional manifolds. F\'{e}lix--Thomas \cite{F-Th} (see also \cite{F-Ta}) constructed the model for rational cohomology of unordered configuration spaces of closed orientable even dimension manifolds. More recently, the identification was established in full generality by the Knudsen in \cite{Kn} using the theory of factorization homology developed by Ayala–Francis [2]. We will restrict our attention to the case
of orientable even dimensional manifolds.

Let us introduced some notations. Consider two graded vector spaces $$V^{*}=H^{-*}_{c}(M;\mathbb{Q})[d],\quad W^{*}=H_{c}^{-*}(M;\mathbb{Q})[2d-1]:$$
where
$$ V^{*}=\bigoplus_{i=0}^{d}V^{i},\quad W^{*}=\bigoplus_{j=d-1}^{2d-1}W^{j}.$$
We choose bases in $V^{i}$ and $W^{j}$ as 
$$V^{i}=\mathbb{Q}\langle v_{i,1},v_{i,2},\ldots\rangle,\quad W^{j}=\mathbb{Q}\langle w_{j,1},w_{j,2},\ldots\rangle$$
(the degree of an element is marked by the first lower index; $x_{i}^{l}$ stands for the product $x_{i}\wedge x_{i}\wedge\ldots\wedge x_{i}$ of $l$-factors). Always we take $V^{0}=\mathbb{Q}\langle v_{0}\rangle$. Now consider the graded algebra
$$ \Omega^{*,*}_{k}(M)=\bigoplus_{i\geq 0}\bigoplus_{\omega=0}^{\left\lfloor\frac{k}{2}\right\rfloor}
\Omega^{i,\omega}_{k}(M)=\bigoplus_{\omega=0}^{\left\lfloor\frac{k}{2}\right\rfloor}\,(Sym^{k-2\omega}(V^{*})\otimes Sym^{\omega}(W^{*})) $$
where the total degree $i$ is given by the grading of $V^{*}$ and $W^{*}$. We called $\omega$ is a weight grading. The differential $\partial:Sym^{2}(V^{*})\rightarrow W^{*}$ is defined as a coderivation by the equation 
$$\partial(s^{d}a\wedge s^{d}b)=(-1)^{(d-1)|b|}s^{2d-1}(a\cup b),$$ where $$\cup\,:H^{-*}_{c}(M;\mathbb{Q})^{\otimes2}\rightarrow H^{-*}_{c}(M;\mathbb{Q})$$
(here $H^{-*}_{c}$ denotes compactly supported cohomology of $M$). The degree of $\partial$ is $-1.$ It can be easily seen that $s^{d}a,\,s^{d}b\in V^{*}$ and $s^{2d-1}(a\cup b)\in W^{*
}.$ The differential $\partial$ extends over  $\Omega^{*,*}_{k}(M)$ by co-Leibniz rule. By definition the elements in $V^{*}$ have length 1 and weight 0 and the elements in $W^{*}$ have length 2 and weight 1. By definition of differential, we have 
$$\partial:\Omega^{*,*}_{k}(M)\longrightarrow\Omega^{*-1,*+1}_{k}(M).$$

\begin{theorem}
	If $d$ is even, $H_{*}(C_{k}(M);\mathbb{Q})$  is isomorphic to the homology of the complex 
	$$ (\Omega^{*,*}_{k}(M),\partial).$$
\end{theorem}
For a closed manifold the compactly supported cohomology is the ordinary cohomology. In this case the two graded vector spaces are $$V^{*}=H_{-*}(M;\mathbb{Q})[d],\quad W_{*}=H_{-*}(M;\mathbb{Q})[2d-1].$$
Now, we will define the dual complex of $(\Omega^{*,*}_{k}(M),\partial).$ First, we define a dual differential $D$ on $\Omega^{*,*}_{k}(M).$ The dual differential is defined as $$D|_{V^{*}}=0,\quad D|_{W^{*}}:\,W^{*}\simeq H_{*}(M;\mathbb{Q})\xrightarrow{\Delta} Sym^{2}(V^{*})\simeq  Sym^{2}(H_{*}(M;\mathbb{Q})),$$
where
$\Delta$ is diagonal comultiplication corresponding to cup product. By definition of differential, we have 
$$D:\Omega^{*,*}_{k}(M)\longrightarrow\Omega^{*+1,*-1}_{k}(M).$$
\begin{theorem}
	If $d$ is even and $M$ is closed, then $H^{*}(C_{k}(M);\mathbb{Q})$  is isomorphic to the cohomology of the complex 
	$$ (\Omega^{*,*}_{k}(M),D).$$
\end{theorem}
The Chevalley–Eilenberg complex naturally appears in the study of the configuration spaces of manifolds. Important examples of its appearance include the work of B\"{o}digheimer--Cohen--Taylor \cite{B-C-T} (see also \cite{Co}) and F\'{e}lix--Thomas \cite{F-Th}, building on McDuff’s foundational work \cite{MD}; the
work of F\'{e}lix--Tanr\'{e} \cite{F-Ta} following Totaro \cite{T}; and the work of Knudsen \cite{Kn}, building on the work of Ayala--Francis \cite{AF}.

\section{Reduced Chevalley–Eilenberg complex}
In this section, we define an acyclic subcomplex of $(\Omega^{*,*}_{k}(M),D).$ 
\begin{theorem}\label{Acyclic}
Let $M$ be a closed orientable manifold of dimension $d.$ The subspace $$\Omega_{k-2}^{*,*}(M).(v_{d}^{2}, w_{2d-1})< \Omega_{k}^{*,*}(M)$$ is acyclic for $k\geq 2.$
\end{theorem}
\begin{proof}
Let $M$ is closed and orientable. An element in $\Omega_{k-2}^{*,*}(M).(v_{d}^{2}, w_{2d-1})$ has a unique expansion $v_{d}^{2}A+Bw_{2d-1},$ where $A$ and $B$ have no monomial containing $w_{2d-1}.$ The operator $$h(v_{d}^{2}A+Bw_{2d-1})=Bv_{d}^{2}$$ gives a homotopy $id\simeq 0.$

\end{proof}
We denote the reduced complex $(\Omega_{k}^{*,*}(M)/\Omega_{k-2}^{*,*}(M).(v_{d}^{2}, w_{2d-1}),D_{\text{induced}})$ by 
$$({}^{r}\Omega_{k}^{*,*}(M),D).$$
\begin{corollary}\label{reduced}
If $k\geq2$ and $M$ is closed orientable, then we have an isomorphism $$H^{*}({}^{r}\Omega_{k}^{*,*}(M),D)\cong H^{*}(C_{k}(M)).$$
\end{corollary}
\begin{remark}
The particular case $({}^{r}\Omega_{k}^{*,*}(\mathbb{CP}^{m}),D)$ of reduced complex is already obtained by author in \cite{Y}. 
\end{remark}
\begin{remark}
If $M$ is not closed then the subspace $\Omega_{k-2}^{*,*}(M).(v_{d}^{2}, w_{2d-1})$ is vanish. 
\end{remark}

%%%%%%%%%%%%%%%%%%%%%%%%%%%%%%%%%%%%%%%%%%%%%%%%%%%%%%%%%%%%%%%%%%%%%%%%%%%%%%%%%%%%%%%%%%%%%%%%%%%%%%
%%%%%%%%%%%%%%%%%%%%%%%%%%%%%%%%%%%%%%%%%%%%%%%%%%%%%%%%%%%%%%%%%%%%%%%%%%%%%%%%%%%%%%%%%%%%%%%%%%%%%%%%
\section{Proofs of main theorems}
In this section, we give the proofs of Theorems \ref{theoremmain1} and \ref{theoremmain2}. A space $X$ is $r$-connected if $\pi_{i}(X)=0$ for $0\leq i\leq r.$ The connectivity of $X;$ $conn(X),$ is the largest integer with such property. The connectivity of contractible spaces is infinite. Kallel \cite{Ka} proved the following result.
\begin{theorem}\label{Kallel}
Let $M$ be a compact orientable even dimensional manifold and $k\geq 2.$ we have
$$ Chdim(C_{k}(M-U)) \leq\begin{cases}
      (d-1)k-r+1, & \mbox{if }\partial M\cup U=\phi\\
      (d-1)k-r, & \mbox{if }\partial M\cup U\neq\phi,\\
   \end{cases}
$$
where $r\geq 0$ is the connectivity of $M$ if $\partial M\cup U=\phi,$ or the connectivity of the quotient $M/M\cup U$ if $\partial M\cup U\neq\phi.$
\end{theorem}
\textit{Proof of Theorems \ref{theoremmain1} and \ref{theoremmain3}.} Let $\partial M\cup U\neq\phi$ and $k\geq1.$ There is no element of degree bigger than $(d-1)k$ in the complex $(\Omega^{*,*}_{k}(M-U),\partial).$ Now, we want to show that $H_{(d-1)k}(C_{k}(M-U);\mathbb{Q})$ is not trivial for $rank (H^{d-1}(C_{k}(M-U);\mathbb{Z}))\neq0.$ Let $$rank (H^{d-1}(C_{k}(M-U);\mathbb{Z}))=n\geq1.$$ We have $dim(V^{d-1})=n\geq1$ and $dim(W^{2d-2})=n\geq1.$ We choose bases in $V^{d-1}$ and $W^{2d-2}$ as 
$$V^{d-1}=\langle v_{d-1,1},\ldots,v_{d-1,n}\rangle,\quad W^{2d-2}=\langle w_{2d-2,1},\ldots,w_{2d-2,n}\rangle.$$ Due to degree reason, we have $\partial(v_{d-1,i})=0$ and $\partial(w_{2d-2,i})=0.$ Here and in the following computations we use the notation
$$\text{I}=\langle v_{d-1,1},v_{d-1,2},\ldots,v_{d-1,n}\rangle,\quad \text{J}=\langle w_{2d-2,1},w_{2d-2,2}\ldots,w_{2d-2,n}\rangle,$$
and
$$\text{J}^{k}=\langle w_{2d-2,i_{1}}w_{2d-2,i_{2}}\ldots w_{2d-2,i_{k}}\,|\,1\leq i_{j}\leq n\rangle.$$
Clearly, $dim(\text{I})=n$ and $dim(\text{J}^{k})=\binom{k+n-1}{n-1}.$\\\\
\textbf{Case 1.} Let $k\geq 2$ even. Consider $\text{J}^{\frac{k}{2}}= \Omega^{(d-1)k,\frac{k}{2}}_{k}(M-U),$ where $$dim\left(\Omega^{(d-1)k,\frac{k}{2}}_{k}(M-U)\right)=\binom{\frac{k}{2}+n-1}{n-1}.$$ Clearly $\partial(\text{J}^{\frac{k}{2}})=0.$ The degree of $\partial$ is -1. Also, there is no element of degree bigger than $(d-1)k.$ Therefore, $\text{J}^{\frac{k}{2}}$ gives the homology classes. Hence, for $k\geq2$ even, we have
$$dim \left(H_{(d-1)k}(C_{k}(M-U);\mathbb{Q})\right)\geq \binom{\frac{k}{2}+n-1}{n-1}\neq0.$$
\textbf{Case 2.} Let $k\geq 1$ odd. Consider $\text{I}\text{J}^{\frac{k-1}{2}}= \Omega^{(d-1)k,\frac{k-1}{2}}_{k}(M-U),$ where $$dim\left(\Omega^{(d-1)k,\frac{k-1}{2}}_{k}(M-U)\right)=n\binom{\frac{k-1}{2}+n-1}{n-1}.$$ Clearly, $\partial(\text{I}\text{J}^{\frac{k-1}{2}})=0.$ The differential $\partial$ has bidegree $(-1,1).$ Moreover, there is no element of degree bigger than $(d-1)k.$ 
Therefore, $\text{I}\text{J}^{\frac{k-1}{2}}$ gives the homology classes. Hence, for $k\geq1$ odd, we have
$$dim \left(H_{(d-1)k}(C_{k}(M-U);\mathbb{Q})\right)\geq n\binom{\frac{k-1}{2}+n-1}{n-1}\neq0.$$\\
We have isomorphism 
$$H^{(d-1)k}(C_{k}(M-U);\mathbb{Q})\cong H_{(d-1)k}(C_{k}(M-U);\mathbb{Q}).$$
If $k\geq 1$ and $M\cup U\neq\phi,$ then from Theorem \ref{Kallel}, we have 
 $ Chdim(C_{k}(M-U))\leq (d-1)k-r.$
Hence, from the previous computations, we have  
$$ Chdim(C_{k}(M-U))= (d-1)k.$$
This completes the proof.
$\hfill \square$\\\\

\textit{Proof of Theorems \ref{theoremmain2} and \ref{theoremmain4}.} Let $M$ is closed and $k\geq3.$ There is no element of degree bigger than $(d-1)k+1$ in the reduced complex $({}^{r}\Omega_{k}^{*,*}(M),D).$  Now, we want to show that $H^{(d-1)k+1}(C_{k}(M);\mathbb{Q})$ is not trivial for $rank (H^{d-1}(C_{k}(M);\mathbb{Z}))\neq0,1.$ Let $$rank (H^{d-1}(C_{k}(M);\mathbb{Z}))=n\geq2.$$ We have $dim(V^{d-1})=n\geq2$ and $dim(W^{2d-2})=n\geq2.$ We choose bases in $V^{d-1}$ and $W^{2d-2}$ as 
$$V^{d-1}=\langle v_{d-1,1},v_{d-1,2},\ldots,v_{d-1,n}\rangle,\quad W^{2d-2}=\langle w_{2d-2,1},w_{2d-2,2},\ldots,w_{2d-2,n}\rangle.$$ By definition of differential, we have $D(v_{d-1,i})=0$ and $D(w_{2d-2,i})=2v_{d-1,i}v_{d}.$ Here and in the following computations we use the notation
$$\text{I}=\langle v_{d-1,1},v_{d-1,2},\ldots,v_{d-1,n}\rangle,\quad \text{J}=\langle w_{2d-2,1},w_{2d-2,2}\ldots,w_{2d-2,n}\rangle,$$
and
$$\text{J}^{k}=\langle w_{2d-2,i_{1}}w_{2d-2,i_{2}}\ldots w_{2d-2,i_{k}}\,|\,1\leq i_{j}\leq n\rangle.$$
Clearly, $dim(\text{I})=n$ and $dim(\text{J}^{k})=\binom{k+n-1}{n-1}.$\\\\
\textbf{Case 1.} Let $k\geq 3$ odd. Consider $v_{d}\text{J}^{\frac{k-1}{2}}={}^{r}\Omega^{(d-1)k+1,\frac{k-1}{2}}_{k}(M),$ where $$dim\left({}^{r}\Omega^{(d-1)k+1,\frac{k-1}{2}}_{k}(M)\right)=\binom{\frac{k-1}{2}+n-1}{n-1}.$$ We have $D(v_{d}\text{J}^{\frac{k-1}{2}})=0,$ in the reduced complex (Note that $v_{d}^{a\geq2}=0$ in the reduced complex). Also, ${}^{r}\Omega^{*,l\geq \frac{k-1}{2}}_{k}(M)=0.$ The differential $D$ has bidegree $(1,-1).$ Therefore,  $v_{d}\text{J}^{\frac{k-1}{2}}$ gives the cohomology classes. Hence, for $k\geq 3$ odd, we have $$dim \left(H^{(d-1)k+1}(C_{k}(M);\mathbb{Q})\right)\geq \binom{\frac{k-1}{2}+n-1}{n-1}\neq0.$$
\textbf{Case 2.} Let $k\geq 4$ even. Consider $\text{I}v_{d}\text{J}^{\frac{k}{2}-1}={}^{r}\Omega^{(d-1)k+1,\frac{k}{2}-1}_{k}(M)$ and $\text{J}^{\frac{k}{2}}={}^{r}\Omega^{(d-1)k,\frac{k}{2}}_{k}(M).$ We have $D(\text{I}v_{d}\text{J}^{\frac{k}{2}-1})=0$ in the reduced complex (Note that $v_{d}^{a\geq2}=0$ in the reduced complex). Therefore, we have sub-complex
$$0\rightarrow {}^{r}\Omega^{(d-1)k,\frac{k}{2}}_{k}(M)\rightarrow {}^{r}\Omega^{(d-1)k+1,\frac{k}{2}-1}_{k}(M)\rightarrow 0,$$
where $dim\left({}^{r}\Omega^{(d-1)k,\frac{k}{2}}_{k}(M)\right)=\binom{\frac{k}{2}+n-1}{n-1}$ and $dim\left({}^{r}\Omega^{(d-1)k+1,\frac{k}{2}-1}_{k}(M)\right)=n\binom{\frac{k}{2}+n-2}{n-1}.$ 
Clearly, for $k>2$ even, and $n>1,$ we have $$n\binom{\frac{k}{2}+n-2}{n-1}>\binom{\frac{k}{2}+n-1}{n-1}.$$
This implies that $$dim \left(H^{(d-1)k+1}(C_{k}(M);\mathbb{Q})\right)\geq n\binom{\frac{k}{2}+n-2}{n-1}-\binom{\frac{k}{2}+n-1}{n-1}\neq0,$$ where $k>2$ even, and $n>1.$\\

If $k\geq 2$ and $M\cup U=\phi,$ then from Theorem \ref{Kallel}, we have 
 $ Chdim(C_{k}(M))\leq (d-1)k-r+1.$
Hence, from the previous computations, we have  
$$ Chdim(C_{k}(M))= (d-1)k+1.$$ This completes the proof.
$\hfill \square$

\section{Application to ordered configuration spaces}
As an application of Theorem \ref{theoremmain1} and Theorem \ref{theoremmain2}, we immediately obtain a sharp lower bound for the cohomological dimension of ordered configuration spaces. 
\begin{corollary}\label{cormain1}
Let $M$ be a compact connected orientable even dimensional manifold. Assume $\partial M\cup U\neq\phi$ and $rank(H^{d-1}(M-U;\mathbb{Z}))\neq0.$ Then, for $k\geq1,$ we have $$ Chdim(F_{k}(M-U))\geq (d-1)k.$$
\end{corollary}
\begin{corollary}\label{cormain2}
Let $M$ be a closed connected orientable even dimensional manifold. Assume $rank(H^{d-1}(M;\mathbb{Z}))\neq0, 1.$ Then, for $k\geq3,$ we have $$ Chdim(F_{k}(M))\geq (d-1)k+1.$$
\end{corollary}

\textit{Proof of corollaries \ref{cormain1} and \ref{cormain2}.}
The action of symmetric group $S_{k}$ on $F_{k}(M)$ induces an action on the cohomology group $H^{i}(F_{k}(M)),$ and the (transfer) map for the finite cover $F_{k}(M)\rightarrow C_{k}(M)$ gives an isomorphism $$H^{i}(C_{k}(M);\mathbb{Q})\approx H^{i}(F_{k}(M);\mathbb{Q})^{S_{k}}$$
between the $S_{k}-$invariant part of the cohomology of the cover and the cohomology of the quotient. The cohomology group $H^{i}(C_{k}(M);\mathbb{Q})$ is obtained by the trivial representation of $H^{i}(F_{k}(M);\mathbb{Q})$ (for more details see \cite{C}). It means that if $H^{i}(C_{k}(M);\mathbb{Q})\neq0$ then $H^{i}(F_{k}(M);\mathbb{Q})$ is also not zero. Hence, from Theorems \ref{theoremmain1} and \ref{theoremmain2}, we get the required lower bounds.
$\hfill \square$\\

\begin{remark} 
For a closed orientable surface $S$ different form sphere, we have $$Chdim(F_{k}(S))=k+1.$$
For more details see Theorem 2.1 of \cite{GGM}. In this sense the lower bound in corollary \ref{cormain2} is sharp. It seems that the lower bound in corollary \ref{cormain1} and corollary \ref{cormain2} are the precise formulas for the cohomological dimensions of ordered configuration spaces.
\end{remark}

\noindent\textbf{Acknowledgement}\textit{.} The author gratefully acknowledge the support from the ASSMS, GC university Lahore. This research is partially supported by Higher Education Commission of Pakistan.

\vskip 0,65 true cm



%% References

\vskip 0,65 true cm






% \end{thebibliography}

\null\hfill  Abdus Salam School of Mathematical Sciences,\\
\null\hfill  GC University Lahore, Pakistan. \\\\
\null\hfill Al-Wadood Institute of professional studies, Islamabad.\\
\null\hfill E-mail: {yameen99khan@gmail.com}

\end{document}